\documentclass[a4paper,oneside,11pt]{amsart}
\usepackage{listings}
\usepackage{amsmath}
\usepackage{verbatim}
\usepackage{csquotes}
\usepackage{mathrsfs}
\usepackage[bookmarks=false]{hyperref}
\usepackage{nameref}
\usepackage{color}
\usepackage{xpatch}
\usepackage{chngcntr}
\usepackage{apptools}
\usepackage{accents}
\usepackage{enumitem}
\usepackage{cleveref}
\newtheorem{lem}{Lemma}
\newtheorem{prop}{Proposition}
\newtheorem{thm}{Theorem}
\newtheorem*{thm-non}{Theorem}

\newtheorem{cor}{Corollary}
\theoremstyle{definition}
\newtheorem{defn}{Definition}
\newtheorem{rem}{Remark}

\newcounter{numl}
\newcommand{\labelnuml}{\textup{(\roman{numl})}}
\newenvironment{numlist}{\begin{list}{\labelnuml}%
{\usecounter{numl}\setlength{\leftmargin}{0pt}%
\setlength{\itemindent}{2\parindent}%
\setlength{\itemsep}{\smallskipamount}\def
\makelabel ##1{\hss \llap {\upshape ##1}}}}{\end{list}}

\DeclareSymbolFont{script}{U}{eus}{m}{n}
\DeclareSymbolFontAlphabet{\mathscr}{script}
\DeclareMathSymbol{\Wedge}{0}{script}{"5E}
\DeclareMathAlphabet{\mathrmsl}{OT1}{cmr}{m}{sl}
\newcommand{\R}{{\mathbb R}}
\newcommand{\C}{{\mathbb C}}

\newcommand{\Q}{{\mathbb Q}}
\newcommand{\T}{{\mathbb T}}

\newcommand{\cL}{{\mathcal L}}
\newcommand{\cO}{{\mathcal O}}
\newcommand{\Id}{\mathrm{Id}}

\newcommand{\grad}{\mathrm{grad}}

\newcommand{\FS}{{\rm FS}}
\newcommand{\Hilb}{{\rm Hilb}}

\newcommand{\Tr}{{\rm Tr}}
\newcommand{\vol}{\mathrm{vol}}
\newcommand{\hxi}{\hat{\xi}}
\newcommand{\bigslant}[2]{{\raisebox{.2em}{$#1$}\left/\raisebox{-.2em}{$#2$}\right.}}

\usepackage[left=3cm,top=3cm,right=3cm,bottom=3cm,bindingoffset=0.5cm]{geometry}

\begin{document}

\author[A. Lahdili]{Abdellah Lahdili}
\address{Lahdili Abdellah \\ D{\'e}partement de Math{\'e}matiques\\
UQAM\\ C.P. 8888 \\ Succursale Centre-ville \\ Montr{\'e}al (Qu{\'e}bec) \\
H3C 3P8 \\ Canada}
\email{lahdili.abdellah@courrier.uqam.ca}

\title[]{Conformally K\"ahler, Einstein--Maxwell metrics and boundedness of the modified Mabuchi functional}

\begin{abstract} 
We prove that if a compact smooth polarized complex manifold admits in the corresponding Hodge K\"ahler class a conformally K\"ahler, Einstein--Maxwell metric, or more generally, a K\"ahler metric of constant $(\xi, a, p)$-scalar curvature in the sense of \cite{ACGL, lahdili}, then this metric minimizes the $(\xi,a,p)$-Mabuchi functional introduced in \cite{lahdili}. Our method of proof extends the approach introduced by Donaldson \cite{donaldson2,donaldson3} and developed by Li \cite{li} and Sano--Tipler \cite{ST}, via finite dimensional approximations and generalized balanced metrics. As an application of our result and the recent construction of Koca--T{\o}nnesen-Friedman \cite{KoTo16}, we describe the K\"ahler classes on a geometrically ruled complex surface of genus  greater than 2,  which admit conformally K\"ahler, Einstein-Maxwell metrics. 
\end{abstract}

\maketitle

\section{Introduction}  
Let $(X,J,g)$ be a compact K\"ahler manifold of real dimension $2m\geq 4$. A Hermitian metric $\tilde{g}$ is said to be {\it conformally K\"ahler, Einstein--Maxwell} metric (cKEM for short) if there exist a smooth positive function $f$ on $X$ such that $g:=f^{2}\tilde{g}$ is a K\"ahler metric, and  
\begin{enumerate}
\item[(a)] $\xi:=J\grad_g(f)$ is Killing for both $g$ and $\tilde{g}$,
\item[(b)] $\tilde{g}$ has constant scalar curvature, i.e. 
\begin{equation}\label{Scal_-2m+1}
{\rm Scal}_{\tilde g}={\rm const}.
\end{equation}
\end{enumerate}
In dimensions $4$, cKEM metrics have been first introduced and studied in \cite{ambitoric1,lebrun0,lebrun1} whereas a high dimensional extension of the theory within a formal GIT setting was found by Apostolov--Maschler \cite{AM}. A number of recent existence results appear in \cite{FO,FO1,KoTo16,lahdili}.

The terminology refers to the fact \cite{AM} that the trace-free part of the Ricci tensor of a cKEM metric $\tilde g$ can be written as $\omega^{-1} \circ \varphi_0$, where $\omega$ is the (closed) K\"ahler form of $(g, J)$ and $\varphi_0$ is a co-closed primitive $(1,1)$-form. Notice that in dimension $4$,  both $\omega$ and $\varphi_0$ are then harmonic as being self-dual and anti-selfdual $2$-forms, respectively, so that a cKEM metric is an example of a Riemannian metric satisfying the Einstein-Maxwell equations with cosmological constant in general relativity, see \cite{DKM, GP}. We also notice that if we take in the above definition $f\equiv 1$, then $\tilde g =g$ is simply a K\"ahler metric of constant scalar curvature (cscK), so that the theory of cKEM metrics naturally generalizes the theory of cscK metrics which is the subject of most active current research. It is thus natural to try to extend known results about cscK metrics to the cKEM setting.

\smallskip

One can put cKEM metrics in the general framework of K\"ahler metrics with constant $(\xi,a,p)$-scalar curvature, recently introduced in \cite{ACGL,lahdili}. Denote by ${\rm Aut}_{\rm red}(X)$ the {\it reduced automorphism group} of $X$, see e.g. \cite{Gauduchon-book}, i.e. the closed subgroup of complex automorphisms of $X$ whose Lie algebra is the ideal of holomorphic vector fields with zeroes. Let $\xi$ be a fixed real holomorphic vector field with zeros, and suppose that $\xi$ is {\it quasi-periodic}, i.e. the flow of $\xi$ generates a real torus $\T_\xi\subset {\rm Aut}_{\rm red}(X)$. We let $\Omega\in H^{2}_{\rm dR}(X,\R)\cap H^{1,1}(X,\C)$ be a fixed K\"ahler class on $(X,J)$ and denote by $\mathcal{K}^{\xi}_\Omega$ the space of $\T_\xi$-invariant K\"ahler forms $\omega\in\Omega$. It is well-known (see e.g. \cite{Gauduchon-book}) that for any $\omega\in\mathcal{K}^{\xi}_\Omega$ the vector field $\xi$ is Hamiltonian with respect to $\omega$, i.e. $\iota_\xi\omega=-df$ for a smooth function $f$ on $X$ called the {\it Killing potential} of $\xi$. We normalize $f=f_{(\xi,\omega,a)}$ by choosing a real positive constant $a>0$ and requiring $\int_X f_{(\xi,\omega,a)} \vol_\omega=a$. As noticed in \cite[Lemma 1]{AM}, we can choose $a>0$ so that $f_{(\xi,\omega,a)} >0$ for any $\omega\in\mathcal{K}^{\xi}_{\Omega}$. For any real number $p\in\mathbb{R}$ we then define the $(\xi,a,p)$-{\it scalar curvature} of the K\"ahler metric $g=\omega(\cdot,J\cdot)$, $\omega\in\mathcal{K}^{\xi}_{\Omega}$ to be
\begin{equation}\label{Scal_p}
{\rm Scal}_{(\xi,a,p)}(g):=f^{2}_{(\xi,\omega,a)}{\rm Scal}_g-2(p-1)f_{(\xi,\omega,a)} \Delta_g f_{(\xi,\omega,a)}-p(p-1)|\xi|^{2}_g,
\end{equation}
where ${\rm Scal}_g$ is the usual scalar curvature, $\Delta_g$ is the Riemannian Laplacian on functions, and $|\cdot|_g$ is the tensor norm induced by $g$. The point of this definition is that it extends the usual scalar curvature (which corresponds to the case $\xi=0$) and for the special value $p=2m$ it computes the scalar curvature of the conformal metric $\tilde{g}=f^{-2}_{(\xi,\omega,a)} g$, so that cKEM metrics correspond to K\"ahler metrics with constant $(\xi,a,2m)$-scalar curvature (where $\xi= J \grad_g f$ and $a= \int_X f \vol_{\omega}$). Other values of the parameter $p$ have interesting geometric interpretations too, see e.g. \cite{ACGL, LU}.\\

\smallskip
 
For a fixed K\"ahler form $\omega\in\mathcal{K}^{\xi}_\Omega$ we denote by $\mathcal{K}^{\xi}_\omega$ the space of smooth $\T_\xi$-invariant K\"ahler potentials with respect to $\omega$:
\begin{equation*}
\mathcal{K}^{\xi}_\omega=\{\phi\in C^{\infty}(X,\R)^{\xi}\,\mid\,\omega_\phi:=\omega+dd^{c}\phi>0\},
\end{equation*}
where $C^{\infty}(X,\R)^{\xi}$ is the space of $\T_\xi$-invariant functions. Following \cite{lahdili}, the K\"ahler metrics with constant $(\xi,a,p)$-scalar curvature are critical points of the $(\xi,a,p)$-Mabuchi energy $\mathcal{M}_{(\xi,a,p)}:\mathcal{K}^{\xi}_\omega\rightarrow\R $, defined by 
\begin{equation}\label{Mabuchi}
\begin{cases} \left({\bold d}\mathcal{M}_{(\xi,a,p)}\right)_{\phi}(\dot{\phi})=&{\displaystyle-\int_X\frac{\dot{\phi}\,\accentset{\circ}{{\rm Scal}}_{(\xi,a,p)}(\phi)}{f^{p+1}_{(\xi,\phi,a)}}\vol_{\phi}},\\ 
\mathcal{M}_{(\xi,a,p)}(\omega)=0,
\end{cases}
\end{equation}
for all $\dot{\phi}\in T_\phi \mathcal{K}^{\xi}_\omega\cong C^{\infty}(X,\R)^{\xi}$. In the above formula $f_{(\xi,\phi,a)}$ is the Killing potential of $\xi$ with respect to $\omega_\phi$ (normalized by the constant $a$ as above), $\vol_\phi$ denotes the volume form $\frac{\omega^{m}_\phi}{m!}$, and
\begin{equation*}
\accentset{\circ}{{\rm Scal}}_{(\xi,a,p)}(\phi):={\rm Scal}_{(\xi,a,p)}(g_\phi)-c_{(\Omega,\xi,a,p)},
\end{equation*}
where $c_{(\Omega,\xi,a,p)}$ is the weighted average of ${\rm Scal}_{(\xi,a,p)}(\omega)$, i.e.
\begin{equation}\label{top-const}
c_{(\Omega,\xi,a,p)}:=\dfrac{\displaystyle\int_X\frac{{\rm Scal}_{(\xi,a,p)}(\omega)}{f^{p+1}_{(\xi,a,\omega)}}\vol_{\omega}}{\displaystyle\int_X\frac{1}{f^{p+1}_{(\xi,a,\omega)}}\vol_{\omega}},
\end{equation}
which is independent of the choice of $\omega\in\mathcal{K}^{\xi}_\Omega$, see \cite{AM,lahdili}. Notice that if $\xi=0$, then $\mathcal{M}_{(\xi,a,p)}$ is the usual Mabuchi energy whose critical points are the cscK metrics.\\

\smallskip

A corner-stone result in the theory of cscK metrics is the fact such a metric minimizes the Mabuchi energy in the corresponding K\"ahler class \cite{B-B, CT, donaldson3, li,  ST}. In this paper, we establish an analogous result concerning K\"ahler metrics of constant $(\xi, a, p)$-scalar curvature in a given Hodge K\"ahler class $\Omega=2\pi c_1(L)$, where $L$ is a positive holomorphic line bundle on $X$. Our proof will follow Donladson's method \cite{donaldson3} via approximations with balanced metrics in the case when ${\rm Aut}_{\rm red}(X)$ is discrete, and its ramifications found by C. Li \cite{li} and Sano--Tipler \cite{ST} in the case of a general smooth polarized complex manifold.

\begin{thm}\label{Mabuchi-bounded}
Let $(X,L)$ be a compact smooth polarized complex manifold and $\xi$ a quasi-periodic real holomorphic vector field on $X$, generating a real torus $\T_\xi$ in ${\rm Aut}_{\rm red}(X)$. Then, the $(\xi,a,p)$-Mabuchi energy of the K\"ahler class $\Omega = 2\pi c_1(L)$ attains its minimum at a $\T_\xi$-invariant $(\xi,a,p)$-constant scalar curvature K\"ahler metric in $\Omega$.
\end{thm}
As we have already observed, the above result extends \cite{li} to the general cKEM setting. One can also define a notion of $(\xi,a,p)$-extremal K\"ahler metric (see \cite{lahdili}), by requiring that the $(\xi,a,p)$-scalar curvature defined by \eqref{Scal_p} is a Killing potential for the K\"ahler metric $g$, rather than being just a constant. One can also introduce (see \cite{lahdili}) a relative version of the functional $\mathcal{M}_{(\xi,a,p)}$, reminiscent to the relative Mabuchi energy describing the extremal K\"ahler metrics \cite{Gauduchon-book}. Presumably, the methods of this article can be adapted to extend \Cref{Mabuchi-bounded} to the relative $(\xi,a,p)$-extremal case, along the lines of \cite{ST}, but we shall discuss this and sum related questions in a forthcoming work.

One might also hope to extend \Cref{Mabuchi-bounded} beyond the polarized case. Indeed, in the cscK case such an extension have been found via a deep result of Berman–-Berndtsson \cite{B-B} on the convexity and boundedness of the Mabuchi functional. We expect that along the method of [11] similar properties can possibly be established for the  the functional $\mathcal{M}_{(\xi, a, p)}$, but the details go beyond the scope of the present article.\\

\smallskip

It is observed in \cite{ACGT} that the analogous result of \Cref{Mabuchi-bounded} about extemal K\"ahler metrics can be used in order to classify, on a geometrically ruled complex surface,  the K\"ahler classes admitting extremal K\"ahler metrics. Similarly, combining \Cref{Mabuchi-bounded} above with the construction of \cite{KoTo16}, we show

\begin{cor}\label{c:main}
Let $X=\mathbb{P}(\mathcal{O}\oplus \mathcal{L})\to C$ be a geometrically ruled complex surface over a compact complex curve $C$ of genus ${\bf g}\geq 2,$ where $\mathcal{L}$ is a holomorphic line bundle over C of positive degree, and $\Omega_\kappa= 2\pi \left( c_1(\mathcal{O}(2)_{\mathbb{P}(\mathcal{O}\oplus \mathcal{L})}) + (1+\kappa)\cdot c_1(\mathcal{L})\right)$, $\kappa>1$ is the effective parametrization of the K\"ahler cone of $X$, up to positive scales, see e.g. \cite{at, fujiki}. Then,
\begin{enumerate}
\item[\rm (a)] There exists a real constant $\kappa_0(X)>1$, such that for each $\kappa>\kappa_0(X)$, $\Omega_\kappa$ admits a cKEM  metric, see \cite{KoTo16};
\item[\rm (b)] For any $\kappa\in (1,\kappa_0(X)]$, $\Omega_\kappa$ does not admit a cKEM metric.
\end{enumerate}
\end{cor}

\subsection*{Outline of the paper} 
Section 2 contains our main new technical observation (Corollary 2) which, in the polarized case $(X, L)$, identifies the $(\xi, a, p)$-scalar curvature \eqref{Scal_p} with a coefficient in the asymptotic expansion of a suitable $\hxi$-equivariant weighted Bergman kernels associated to the vector spaces of holomorphic sections $\mathcal{H}_k=H^0(X,L^k)$, where $\hxi$ denotes the lifted real holomorphic vector field on $L$, associated to the data $(\xi, a)$. This  result, which is based on the general theory of Berezin--Toeplitz operators developed by Charles \cite{laurent} and Zelditch--Zhou \cite{ZZ}, extends the famous Bergman kernel expansion theorem \cite{BBS, catalin, DLM, Lu, Ruan, Tian-Berg, Zel} to a larger family of equivariant Bergman operators, and allows us to define a suitable notion of  a {\it $(\hxi, p)$-balanced} Fubini--Study metric on  $\mathbb{P}(\mathcal{H}_{k}^*)$ and hence of {\it $(\hat \xi, p)$-balanced} K\"ahler metrics on $X$ (via the Kodaira embedding). We then show in \Cref{prop-bal} that, similarly to \cite{donaldson3}, if a sequence of $(\hxi,p)$-balanced metrics on $X$ converges as $k \to \infty$ to a smooth K\"ahler metric $\omega\in 2\pi c_1(L)$, then $\omega$ must have constant $(\xi,a,p)$-scalar curvature.  In Section 3, we give the proof of \Cref{Mabuchi-bounded}, closely following the method of \cite{donaldson3, li, ST}.  The main new ingredient here is the introduction of suitable functionals on the finite dimensional spaces of Fubini--Study metrics on $\mathbb{P}(\mathcal{H}^{*}_k)$, quantizing the $(\hxi,p)$-Mabuchi energy, and whose minima are the $(\hxi, p)$-balanced metrics (Proposition 2). With these in place, the proof of \Cref{Mabuchi-bounded} is not materially different than the arguments in \cite{donaldson3, li, ST}. In the final Section 4, we give the proof of \Cref{c:main}. It contains three main ingredients. The first one is the definition in \cite{KoTo16} of a polynomial  $P_{\kappa}(x)$ of degree $\leq 4$, and a constant $a>0$,  associated to each normalized K\"ahler class $\Omega_{\kappa}$ on $X$, such that if $P_{\kappa}(x)>0$ on $(-1,1)$, it defines, via the Calabi construction, a  K\"ahler metric of constant $(\xi, a, 4)$-scalar curvature in $\Omega_\kappa$, where $\xi$ is the generator of the $\mathbb{S}^1$-action on $X$ by multiplications on $\mathcal{O}$. As shown in \cite{KoTo16}, this yields the existence part (a) of \Cref{c:main}. The second ingredient, taken from \cite{amt, acgt}, is an expression of the $(\xi, a, 4)$-Mabuchi energy in terms of $P_{\kappa}(x)$: \Cref{Mabuchi-bounded} above then yields the non-existence of cKEM metrics associated to $\Omega_\kappa$ in the case when $\kappa$ is rational and $P_\kappa(x)$ is negative somewhere on $(-1,1)$. In order to deal with the limiting cases, i.e. when $\kappa$ is irrational or $P_\kappa(x)$ has a double root in $(-1,1)$, we use the stability result \cite[Theorem 2]{lahdili}.

\section*{Acknowledgement} 
I am grateful to the anonymous referees for their constructive criticism which lead to improvements of the presentation. I would like to thank my supervisor V. Apostolov for his guidance and invaluable advice. I am grateful to V. Apostolov, G. Maschler and C. T{\o}nnesen-Friedman for giving me access to \cite{amt} and allowing me to use their computation of the modified Mabuchi energy (\Cref{amt} below) in order to establish \Cref{c:main}. 

\section{Equivariant Bergman kernels and $(\hat{\xi},p)$-balanced metrics}

In this section, we specialize to the case of a compact smooth polarized complex manifold $(X,L)$, endowed with the K\"ahler class $\Omega = 2\pi c_1(L)$. We denote by  $\hxi$ the lift to $L$ of the holomorphic vector field $\xi$, by using the normalization constant $a$. Our main observation is the definition of a class of $\hxi$-equivariant weighted Bergman kernels on the vector spaces $\mathcal{H}_k := H^0(X, L^k)$ of holomorphic sections of $L^k$, having the $(\xi, a, p)$-scalar curvature \eqref{Scal_p} as a coefficient in their asymptotic expansions as $k \to \infty$ (see \Cref{rho-exp} below). This will allow us to introduce the notion of a $(\hxi, p)$-balanced Fubini--Study metric on $\mathbb{P}(\mathcal{H}^{*}_k)$, and thus of a {\it $(\hxi, p)$-balanced K\"ahler metric} in $\Omega = 2\pi c_1(L)$,  such that   if a sequence of $(\hxi, p)$-balanced metrics on $X$ smoothly converges as $k \to \infty$ to a smooth K\"ahler metric $\omega \in 2\pi c_1(L)$, then $\omega$ must have constant $(\xi, a, p)$-scalar curvature (\Cref{prop-bal}).

\subsection{Equivariant Bergman kernels}
Let $(X,L)$ be a polarized manifold with the K\"ahler class $\Omega=2\pi c_1(L)$, $\xi$ a quasi-periodic holomorphic vector field generating a real torus $\T_\xi\subset {\rm Aut}_{\rm red}(X)$, and $h$ a Hermitian metric on $L$ whose Chern curvature $2$-form is $F_h=\omega\in\mathcal{K}^{\xi}_\Omega$. We identify the space of Hermitian metrics $h_\phi:=e^{-2\phi}h$, $\phi\in C^{\infty}(X,\mathbb{R})^{\xi}$ with positive curvature forms $F_{h_\phi}>0$, with the space $\mathcal{K}^{\xi}_\omega$.

Let $a>0$ be a normalization constant such that for all $\phi\in\mathcal{K}^{\xi}_\omega$ we have $\int_X f_{(\xi,\phi,a)}\vol_\phi=a$ and $f_{(\xi,\phi,a)}>0$ on $X$ (see \cite[Lemma 1]{AM}). Once the constant $a>0$ is fixed, for any $\phi\in\mathcal{K}^{\xi}_{\omega}$ one can consider the lift $\hat{\xi}$ of $\xi$ to the total space of $L$, given by 
\begin{equation}\label{hxi}
\hxi=\xi^{H}-f_{(\xi,\phi,a)}\partial_\theta,
\end{equation}
where $\xi^{H}$ denotes the horizontal lift on $L$ of $\xi$ with respect to the Chern connection $\nabla^{\phi}$ of $h_\phi$ on $L$, and $\partial_\theta$ stands for the vector field on $L$ generated by rotations on each fiber. It follows from the arguments in \cite[Lemma 1]{AM} that with our normalization for $f_{(\xi,\phi,a)}$, we have
\begin{equation*}
f_{(\xi,\phi,a)}=f_{(\xi,\omega,a)}+d^{c}\phi(\xi),
\end{equation*}
showing that $\hxi$ defined by \eqref{hxi} is independent of the choice of $\phi\in\mathcal{K}^{\xi}_\omega$. Since $\hxi$ is uniquely determined by $(\xi,a)$, we replace the subscript $(\xi, a)$ with $\hxi$.\\

\smallskip

Let $\Phi^{\xi}_t$ (resp. $\Phi^{\hxi}_t$) denote the flow of $\xi$ (resp. $\hxi$). For any smooth section $s$ of $L$ we define
\begin{equation*}
\left(A_{\hxi}s\right)(x):=-\sqrt{-1}\left.\frac{d}{dt}\right|_{t=0}\Phi^{\hxi}_t\left(s\left(\Phi^{\xi}_{-t}(x)\right)\right).
\end{equation*}
It is well-known that (see e.g. \cite[Proposition 8.8.2]{Gauduchon-book})
\begin{equation}\label{A}
A_{\hxi}=-\sqrt{-1}\nabla^{\phi}_{\xi}+f_{(\hxi,\phi)}.
\end{equation}
For a positive function $\Psi\in C^{\infty}(\R)$, we consider the following inner product on the space $C^{\infty}(X,L)$ of smooth sections of $L$: 
\begin{equation}\label{prod-scal}
\langle s, s^{\prime}\rangle_{(\Psi,\hxi,\phi)}:=\int_X(s,s^\prime)_\phi \Psi(f_{(\hxi,\phi)})\vol_\phi,
\end{equation}
where $(\cdot,\cdot)_\phi$ stands for the pointwise Hermitian product $h_\phi$ on $C^{\infty}(X,L)$. A straightforward calculation shows that $A_{\hxi}$ is a Hermitian operator with respect to $\langle\cdot,\cdot\rangle_{(\Psi,\hxi,\phi)}$, preserving the finite dimensional subspace of holomorphic sections $\mathcal{H}:=H^0(X,L)$. We denote by $\Lambda(\hxi)$ the spectrum of $(A_{\hxi})_{\mid \mathcal{H}}$, and let $f_{(\hxi,\phi)}(X)=[a_0,a_1]$  be the (fixed) image of $X$ under the normalized Killing potential of $\xi$ (see \cite[Lemma 1]{AM}). For an eigensection $s\in\mathcal{H}$ corresponding to the eigenvalue $\lambda\in\Lambda(\hxi)$, we have by \eqref{A}
\begin{equation*}
\lambda\cdot|s|_{\phi}^{2}(x)=-\sqrt{-1}\partial|s|^{2}_{\phi}(\xi)+f_{(\hxi,\phi)}(x)|s|_{\phi}^{2}(x),
\end{equation*}
for all $x\in X$. At a point $x_0\in X$ where $|s|^{2}_{h_\phi}$ attains its maximum, we have $\lambda=f_{(\hxi,\phi)}(x_0)\in[a_0,a_1]$, showing that $\Lambda(\hxi)\subset[a_0,a_1]$. For a smooth function $\Phi\in C^{\infty}(\R)$, we can thus define $\Phi(A_{\hxi})$ using the spectral theorem. 
\begin{defn}\cite{B-N, gabor, ZZ}
For a $\langle \cdot,\cdot\rangle_{(\Psi,\hxi,\phi)}$-orthonormal basis $\{s_i\mid i=0,\cdots,N\}$ of $\mathcal{H}$, the $(\Psi,\Phi,\hxi)$-equivariant Bergman kernel of the Hermitian metric $\phi\in \mathcal{K}^{\xi}_\omega$ is the smooth function on $X$ defined by
\begin{equation}\label{berg}
B_{(\Psi,\Phi,\hxi)}(\phi)(x):=\Psi(f_{(\hxi,\phi)}(x))\sum_{i=0}^{N}\left(\Phi(A_{\hxi})s_i(x),s_i(x)\right)_{\phi}.
\end{equation}
\end{defn}
Clearly \eqref{berg} is independent of the choice of a $\langle \cdot,\cdot\rangle_{(\Psi,\hxi,\phi)}$-orthonormal basis of $\mathcal{H}$. 
\begin{comment}
Let $\Pi_{(\Psi,\hxi,\phi)}:L^{2}(X,L)\rightarrow \mathcal{H}$ denote the orthogonal projection with respect to the inner product $\langle\cdot,\cdot\rangle_{(\Psi,\hxi,\phi)}$.\\
The function $B_{(\Psi,\Phi,\hxi)}(\phi)$ is the restriction to the diagonal $\{x=x^{\prime}\}\subset X\times X$ of the Schwartz reproducing kernel $K_{(\Psi,\Phi,\hxi)}^{\phi}(x,x^{\prime})\in L_x\otimes L_{x^{\prime}}^{\star}$ of the operator $\Pi_{(\Psi,\hxi,\phi)}\Phi(A_{\hxi})\Pi_{(\Psi,\hxi,\phi)}$, given by  
\begin{equation*}
K_{(\Psi,\Phi,\hxi)}^{\phi}(x,x^{\prime})=\Psi(f_{(\hxi,\phi)}(x^{\prime}))\sum_{i=0}^{N}\Phi(A_{\hxi})(s_i)(x)\otimes s_i(x^{\prime})^{\star},
\end{equation*}
where  $\{s_i\mid i=0,\cdots,N\}$ is a $\langle \cdot,\cdot\rangle_{(\Psi,\hxi,\phi)}$-orthonormal basis of $\mathcal{H}$.
\end{comment}
\begin{comment}
\begin{equation*}
\left(\Pi_{(\Psi,\hxi,\phi)}\Phi(A_{\hxi})\Pi_{(\Psi,\hxi,\phi)}s\right)(x)=\int_X K_{(\Psi,\Phi,\hxi)}^{\phi}(x,x^{\prime})\cdot s(x^{\prime})\Psi(f_{(\hxi,\phi)}(x^{\prime}))\vol_{\phi}(x^{\prime}).
\end{equation*}
\end{comment}
\\
For $k\in\mathbb{N}$, we denote by $A^{(k)}_{\hxi}$ the induced operator on $C^{\infty}(X,L^{k})$, and consider for $\phi\in \mathcal{K}^{\xi}_\omega$ the $(\Psi,\Phi,\hxi)$-equivariant Bergman kernel of the metric $h^{k}_\phi$ on $L^{k}$:
\begin{equation}\label{Bergman}
B_{(\Psi,\Phi,\hxi)}(k\phi)=\Psi(f_{(\hxi,\phi)})\sum_{i=0}^{N_k}\left(\Phi\left(\frac{1}{k}A_{\hxi}^{(k)}\right)s_i,s_i\right)_{k\phi}.
\end{equation}
where $\{s_i\mid i=0,\cdots,N_k\}$ is a $\langle \cdot,\cdot\rangle_{(\Psi,\hxi,k\phi)}$-orthonormal basis of $\mathcal{H}_k:=H^0(X,L^{k})$.\\

\smallskip

Asymptotic expansions of \eqref{Bergman} for $k\gg 1$ are known to exist in many special cases, see e.g. \cite{B-N, gabor, ma-ma}. In its full generality, needed for the applications in this paper, such an expansion has been shown to exist in \cite{laurent, ZZ}. The result below is a direct corollary of the latter works. In particular, using \cite[Proposition 12]{laurent} we can determine explicitly the second coefficient of the expansion. We note that the coefficient appearing in \cite[Proposition 7.12]{gabor} seems to be incomplete.

\begin{thm}\label{TYZ}\cite{laurent, ZZ}
The $(\Psi,\Phi,\hxi)$-equivariant Bergman kernel of the Hermitian metric $h^{k}_\phi$ on $L^{k}$, for $\phi\in\mathcal{K}^{\xi}_\omega$ admits an asymptotic expansion when $k\gg 1$, given by
\begin{equation*}
(2\pi)^{m}B_{(\Psi,\Phi,\hxi)}(k\phi)=\Phi(f_{(\hxi,\phi)})+\frac{1}{k}S_{(\Psi,\Phi,\hxi)}(\phi)+\mathcal{O}\left(\frac{1}{k^{2}}\right),
\end{equation*}
where
\begin{align*}
S_{(\Psi,\Phi,\hxi)}(\phi)=&\frac{1}{4}\Phi(f_{(\hxi,\phi)}){\rm Scal}(\phi)+\frac{1}{2}\Phi(f_{(\hxi,\phi)})(\log \Psi)^{\prime}(f_{(\hxi,\phi)})\Delta_\phi(f_{(\hxi,\phi)})\\&+\left[\frac{1}{4}\Phi^{\prime\prime}(f_{(\hxi,\phi)})-\frac{1}{2}\Phi^{\prime}(f_{(\hxi,\phi)})(\log\Psi)^{\prime}(f_{(\hxi,\phi)})-\frac{1}{2}\Phi(f_{(\hxi,\phi)})(\log\Psi)^{\prime\prime}(f_{(\hxi,\phi)})\right]|\xi|^{2}_\phi.
\end{align*}
Moreover, the above expansion holds in $C^{\infty}$, i.e. for any integer $\ell\geq 0$ there exist a constant $C_\ell>0$ such that,
\begin{equation*}
\left\Vert(2\pi)^{m}B_{(\Psi,\Phi,\hxi)}(k\phi)-\Phi(f_{(\hxi,\phi)})-\frac{1}{k}S_{(\Psi,\Phi,\hxi)}(\phi)\right\Vert_{C^{\ell}}\leq \frac{C_{\ell}}{k^{2}}
\end{equation*}
\end{thm}
\begin{proof}
The above theorem is an application of \cite[Proposition 12]{laurent} and we refer the reader to this work for the relevant definitions we use below.

We consider the $\langle\cdot,\cdot\rangle_{(\Psi,\Phi,\hxi)}$-self-adjoint Toeplitz operator $T_k:=\Pi_{(\Psi,\hxi,k\phi)}A_{\hxi}\Pi_{(\Psi,\hxi,k\phi)}$ and we denote by $T_k(x)$ the restriction to the diagonal $\{x=x^{\prime}\}$ of its Schwartz kernel i.e. 
\begin{equation*}
T_k(x)=\Psi(f_{(\hxi,\phi)}(x))\sum_{i=0}^{N_k}\left(\frac{1}{k}(A_{\hxi}^{(k)}s_i)(x),s_i(x)\right)_{k\phi}.
\end{equation*}
By a straightforward calculation using \eqref{A} and the formula $\nabla^{k\phi}_{\xi}s_i=(\partial\log|s_i|^{2}_{k\phi})(\xi)s_i$ we have:
\begin{equation*}
T_k(x)=-\frac{\sqrt{-1}}{k}\left(\partial\Pi_{(\Psi,\hxi,k\phi)}\right)(\xi)+\left[f_{(\hxi,\phi)}(x)+\frac{\sqrt{-1}}{k}\partial\log{(\Psi(f_{(\hxi,\phi)}))}(\xi)\right]\Pi_{(\Psi,\hxi,k\phi)}(x)
\end{equation*}
where $\Pi_{(\Psi,\hxi,k\phi)}(x)$ is given by
\begin{equation*}
\Pi_{(\Psi,\hxi,k\phi)}(x)=\Psi(f_{(\hxi,\phi)}(x))\sum_{i=0}^{N_k}\left|s_i\right|_{k\phi}^{2}(x).
\end{equation*}
Thus, $\Pi_{(\Psi,\hxi,k\phi)}(x)$ admits the following $C^{\infty}$-asymptotic expansion (see e.g. \cite[Theorem 4.1.2]{ma-ma})
\begin{equation}\label{xxx}
\Pi_{(\Psi,\xi,k\phi)}(x)=1+\frac{1}{4k}\left[{\rm Scal}(\phi)+2\Delta_\phi(\log \Psi(f_{(\xi,\phi)}))\right]+\mathcal{O}\left(\frac{1}{k^{2}}\right).
\end{equation}
It follows from \cite[Proposition 3]{laurent} and \eqref{xxx} that the symbol of $T_k$ is given by
\begin{equation*}
\sigma(T_k)=f_{(\hxi,\phi)}+\frac{1}{4}\left[f_{(\hxi,\phi)}{\rm Scal}(\phi)+2f_{(\hxi,\phi)}\Delta_\phi(\log \Psi(f_{(\hxi,\phi)}))-2(\log\Psi)^{\prime}(f_{(\hxi,\phi)})|\xi|^{2}_\phi\right]\hbar+\cdots
\end{equation*}
By \cite[Proposition 12]{laurent} (see also \cite[Proposition 2.1]{ZZ}) $\Phi(T_k)$ is a Toeplitz operator too, with symbol
\begin{equation*}
\sigma\left(\Phi(T_k)\right)=\Phi(f_{(\hxi,\phi)})
+\frac{1}{4}S_{(\Psi,\Phi,\hxi)}(\phi)\hbar+\cdots
\end{equation*}
The result follows.
\end{proof}

In the special case when $\Psi(t):=t^{-p+1}$ ($p\in\mathbb{R}$) and $\Phi(t):=t^{q}$ ($q\in\mathbb{R}$), we denote the associated $(\Psi,\Phi,\hxi)$-equivariant Bergman kernel by $B_{(\hxi,p,q)}(k\phi)$. As a direct corollary of \Cref{TYZ}, we have 
\begin{cor}\label{TYZ-q}
We have the following $C^{\infty}$ asymptotic expansion
\begin{equation*}
(2\pi)^{m}B_{(\hxi,p,q)}(k\phi)=f_{(\hxi,\phi)}^{q}+\frac{1}{k}S_{(\hxi,p,q)}(\phi)+\mathcal{O}\left(\frac{1}{k^{2}}\right),
\end{equation*}
where 
\begin{equation*}
4S_{(\hxi,p,q)}(\phi)=f_{(\hxi,\phi)}^{q}{\rm Scal}(\phi)-2(p-1)f_{(\hxi,\phi)}^{q-1}\Delta_\phi(f_{(\hxi,\phi)})+(q-1)(q+2p-2)f_{(\hxi,\phi)}^{q-2}|\xi|^{2}_\phi.
\end{equation*}
\end{cor}
\subsection{$(\hat{\xi},p)$-balanced metrics}
We denote by $\Lambda_k(\hxi)$ the spectrum of $\frac{1}{k}A^{(k)}_{\hxi}$, and consider the decomposition of $\mathcal{H}_k$ as a direct sum of eigenspaces for $\frac{1}{k}A^{(k)}_{\hxi}$:
\begin{equation}\label{decomposition}
\mathcal{H}_k=\bigoplus_{\lambda\in\Lambda_k(\hxi)}\mathcal{H}_k(\lambda).
\end{equation}
For $\lambda\in\Lambda_k(\hxi)$ let $\mathcal{B}_{\hxi}(\mathcal{H}_k(\lambda))$ denote the space of positive definite Hermitian forms which are $\hat{\xi}$-invariant on $\mathcal{H}_k(\lambda)$, and 
\begin{equation*}
\mathcal{B}_{\hxi}(\mathcal{H}_k):=\bigoplus_{\lambda\in\Lambda_k(\hxi)}\mathcal{B}_{\hxi}(\mathcal{H}_k(\lambda)).
\end{equation*}
For $\Psi(t)=t^{-p+1}$, we denote the inner product \eqref{prod-scal} by $\langle\cdot,\cdot\rangle_{(\xi,k\phi,p)}$.

\begin{defn}\label{Quant-map}
We introduce the following quantization maps:
\begin{itemize}
\item The $(\hxi,p)$-Hilbert map ${\rm Hilb}_{(\hxi,p)}^k:\mathcal{K}^{\xi}_\omega\rightarrow\mathcal{B}_{\xi}(\mathcal{H}_k)$ which associates to every $\T_\xi$-invariant K{\"a}hler potential, the $\hat{\xi}$-invariant Hermitian inner product on $\mathcal{H}_k$, given by 
\begin{equation}\label{Hilb}
{\rm Hilb}_{(\hxi,p)}^k(\phi):=\sum_{\lambda\in\Lambda_k(\hxi)} \dfrac{1}{\lambda(p)} \left(\langle\cdot,\cdot\rangle_{(\xi,k\phi,p)}\right)_{\mid \mathcal{H}_k(\lambda)},
\end{equation}
where for $\lambda\in\Lambda_k(\xi)$, we have set
\begin{equation}\label{lambda-p}
\lambda(p):=\lambda^{-p+1}-\frac{c_{(\Omega,\hxi,p)}}{4k}\lambda^{-(p+1)},
\end{equation}
with $c_{(\Omega,\hxi,p)}$ given by \eqref{top-const}.

\item The $(\hxi,p)$-Fubini--Study map ${\rm FS}^{k}_{(\hxi,p)}:\mathcal{B}_{\hxi}(\mathcal{H}_k)\rightarrow\mathcal{K}^{\xi}_\omega$ given by
\begin{equation*}
{\rm FS}^{k}_{(\hxi,p)}(H):=\frac{1}{2k}\log{\left(\frac{1}{C_k}\sum_{i=0}^{N_k}|s_i|_{h^{k}}^{2}\right)},
\end{equation*}
where $C_k$ is the constant given by:
\begin{align}\label{C_k}
&C_k:=\frac{w_{\hxi}^{-p+1}(L^{k})-\frac{c_{(\Omega,\hxi,p)}}{4k}w_{\hxi}^{-(p+1)}(L^{k})}{\displaystyle \int_X f^{-p+1}_{(\hxi,\omega)}\vol_{k\omega}},\\
&w^{q}_{\hxi}(L^{k}):=\Tr\left(\frac{1}{k}A^{(k)}_{\hxi}\right)^{q}.
\end{align}
and $\{s_i\}$ is an adapted $H$-orthonormal basis of $\mathcal{H}_k$. 
\end{itemize}
\end{defn}
Note that 
\begin{equation*}
\omega_{{\rm FS}^{k}_{(\hxi,p)}(H)}=\Phi^{\star}_k\omega_{{\rm FS},k},
\end{equation*}
where $\Phi_k$ is the Kodaira embedding of $X$ to $\mathbb{P}^{N_k}$ using the basis $\{s_i\}$, and $\omega_{{\rm FS},k}$ is the Fubini-Study metric on $\mathbb{P}^{N_k}$.

\begin{defn}
We denote by $\rho_{(\hxi,p)}(k\phi)$ the Bergman kernel of $\Hilb_{(\hxi,p)}^k(\phi)$, given by
\begin{equation}\label{rho(p)}
\rho_{(\hxi,p)}(k\phi):=f_{(\hxi,\phi)}^{-p+1}\sum_{i=1}^{N_k}|s_i|_{k\phi}^{2},
\end{equation}
where $\{s_i\}$ is a $\Hilb_{(\hxi,p)}^k(\phi)$-orthonormal basis of $\mathcal{H}_k$.
\end{defn}
 We can easily see that $\rho_{(\hxi,p)}(k\phi)$ is independent of the chosen basis, so we can take a basis adapted to \eqref{decomposition}, showing that
\begin{equation}\label{rho=B-B}
\rho_{(\hxi,p)}(k\phi)=B_{(\hxi,p,-p+1)}(k\phi)-\frac{c_{(\Omega,\hxi,p)}}{4k}B_{(\hxi,p,-(p+1))}(k\phi).
\end{equation}
\Cref{TYZ-q} thus yields the following
\begin{lem}\label{rho-exp}
$\rho_{(\hxi,p)}(k\phi)$ admits a $C^{\infty}$ expansion when $k\gg1$
\begin{equation*}
(2\pi)^{m}\rho_{(\hxi,p)}(k\phi)=f_{(\hxi,\phi)}^{-p+1}+\frac{1}{4k}f_{(\hxi,\phi)}^{-(p+1)}\accentset{\circ}{{\rm Scal}}_{(\hxi,p)}(\phi)+\mathcal{O}\left(\frac{1}{k^{2}}\right),
\end{equation*}
in the sense that for any integer $\ell\geq 0$ we have, 
\begin{equation*}
\left\Vert(2\pi)^{m}\rho_{(\hxi,p)}(k\phi)-f_{(\hxi,\phi)}^{-p+1}-\frac{1}{4k}f_{(\hxi,\phi)}^{-(p+1)}\accentset{\circ}{{\rm Scal}}_{(\hxi,p)}(\phi)\right\Vert_{C^{\ell}}\leq \frac{C_{\ell}}{k^{2}}.
\end{equation*}
In particular,
\begin{equation*}
{\rm FS}^{k}_{(\hxi,p)}\circ {\rm Hilb}_{(\hxi,p)}^k(\phi)=\phi+\mathcal{O}(k^{-2}).
\end{equation*}
\end{lem} 
Following \cite{donaldson0, Wang, Zhang}, we give the following definition
\begin{defn}
We say that a metric $\phi\in\mathcal{K}^{\xi}_\omega$ is $(\hxi,p)$-balanced of order $k$ if it satisfies:
\begin{equation*}
{\rm FS}^{k}_{(\hxi,p)}\circ {\rm Hilb}_{(\hxi,p)}^k(\phi)=\phi.
\end{equation*}
\end{defn}

Note that if a metric $\phi\in\mathcal{K}^{\xi}_\omega$ is $(\hxi,p)$-balanced of order $k$ then we have: 
\begin{equation*}
\rho_{(\hxi,p)}(k\phi)=C_k f_{(\hxi,\phi)}^{-p+1},
\end{equation*}
where $C_k$ is the constant given by \eqref{C_k}. Similarly to \cite{donaldson0} we have
\begin{prop}\label{prop-bal}
Let $(\phi_j)_{j\geq0}$ be a sequence in $\mathcal{K}^{\xi}_\omega$ such that every $\phi_j$ is a $(\hxi,p)$-balanced metric of order $j$ and $\phi_j$ converge in $C^{\infty}$ to $\phi$. Then $\omega_\phi$ has a constant $(\hxi,p)$-scalar curvature.
\end{prop}
\begin{proof}
By \Cref{rho-exp} for $k\gg1$,
\begin{equation*}
\left\Vert(2\pi)^{m}\rho_{(\hxi,p)}(k\phi_{j})-f_{(\hxi,\phi_j)}^{-p+1}-\frac{1}{4k}f_{(\hxi,\phi_j)}^{-(p+1)}\accentset{\circ}{{\rm Scal}}_{(\hxi,p)}(\phi_j)\right\Vert_{C^{\ell}}\leq \frac{C_{\ell}}{k^{2}}.
\end{equation*}
Letting $j=k$, we get
\begin{equation}\label{y}
\left\Vert(2\pi)^{m}C_k f_{(\hxi,\phi_k)}^{-p+1} -f_{(\hxi,\phi_k)}^{-p+1}-\frac{1}{4k}f_{(\hxi,\phi_k)}^{-(p+1)}\accentset{\circ}{{\rm Scal}}_{(\hxi,p)}(\phi_k)\right\Vert_{C^{\ell}}\leq \frac{C_{\ell}}{k^{2}}.
\end{equation}
From \eqref{C_k} and \eqref{rho=B-B} we get  
\begin{align*}
(2\pi)^{m}C_k=& \frac{(2\pi)^{m}\displaystyle\int_X\rho_{(\hxi,p)}(h^{k})\vol_{k\omega}}{\displaystyle\int_X f_{(\hxi,\omega)}^{-p+1}\vol_{k\omega}}\\
=&1+\mathcal{O}(k^{-2}).
\end{align*}
Taking a limit when $k$ goes to infinity in \eqref{y}, we obtain that ${\rm Scal}_{(\hxi,p)}(\phi)=c_{(\Omega,\hxi,p)}$.
\end{proof}

\section{Boundedness of the $(\hxi,p)$-Mabuchi energy as an obstruction to the existence of K\"ahler metrics with constant $(\hxi,p)$-scalar curvature}

In this section we prove \Cref{Mabuchi-bounded}, following the method of \cite{donaldson3,li,ST}. To this end, for each $k\gg1$, we introduce appropriate functionals on the finite dimensional  space of Fubini--Study metrics on ${\mathbb P}(\mathcal{H}^{*}_k)$, which when identified with a subspace of $\mathcal{K}^{\xi}_\omega$ via the Kodaira embedding, will quantize the $(\hxi, p)$-Mabuchi functional of $\Omega=2\pi c_1(L)$. Furthermore, following the main ideas of \cite{donaldson3,li,ST}, we will show that the $(\hxi, p)$-balanced metrics are minima of these functionals, and that a K\"ahler metric with constant $(\xi, a, p)$-scalar curvature induces  {\it almost} $(\hxi, p)$-balanced Fubini--Study metrics on ${\mathbb P}({\mathcal H}^{*}_k)$ for $k\gg1$, i.e. they minimize the corresponding functionals up to an error that goes to zero.
\bigskip
\subsection{Quantization of the $(\hat{\xi},p)$-Mabuchi energy}
We start with introducing finite dimensional analogues of the $(\hxi,p)$-Mabuchi energy \eqref{Mabuchi}, given by \eqref{L-Z} on the spaces $\mathcal{B}_{\hxi}(\mathcal{H}_k)$ and ${\rm FS}^{k}_{(\hxi,p)}\left(\mathcal{B}_{\hxi}(\mathcal{H}_k)\right)$ (see \Cref{Quant-map}), respectively, thus a framework for the proof of \Cref{Mabuchi-bounded} along the lines of \cite{donaldson3, li, ST}.\\ 

\smallskip

We introduce the functional $I_{(\hxi,p)}^{k}:\mathcal{B}_{\hxi}(\mathcal{H}_k)\rightarrow \mathbb{R}$ by
\begin{equation}\label{i}
I_{(\hxi,p)}^{k}(H)=\sum_{\lambda\in\Lambda_k(\hxi)}\lambda(p)\log{\left(\det H_\lambda\right)},
\end{equation}
where we recall that $\lambda(p)$ is the expression \eqref{lambda-p}. This functional is a quantization of the Aubin type functional $\mathbb{I}_{(\hxi,p)}^{k}$ on $\mathcal{K}^{\xi}_\omega$, given by its first variation,
\begin{equation}\label{I}
\begin{cases} \left({\bold d} \mathbb{I}_{(\hxi,p)}^{k}\right)_\phi(\dot{\phi})=&2kC_k{\displaystyle \int_X\dot{\phi}f^{-p+1}_{(\hxi,\phi)}\vol_{k\omega_\phi}}
\\ 
\mathbb{I}_{(\hxi,p)}^{k}(\omega)=0
\end{cases}
\end{equation}
where $C_k$ is given by \eqref{C_k}. It is then straightforward to check that
\begin{lem}
1. We have the following expressions for the variations of the Aubin functional $\mathbb{I}_{(\hxi,p)}^{k}$ and its finite dimensional version $I_{(\hxi,p)}^{k}$:
\begin{align}
\begin{split}\label{dI}
\left(\bold{d} \mathbb{I}_{(\hxi,p)}^{k}\right)_\phi(\dot{\phi})=&2kC_k\int_X\dot{\phi}\left(1+\frac{\Delta_{\phi}}{2k}\right)f^{-p+1}_{(\hxi,\phi)}\vol_{k\omega_\phi}\\&-C_k\int_X(d\dot{\phi},d\log f^{-p+1}_{(\hxi,\phi)})_{g_{\phi}}f^{-p+1}_{(\hxi,\phi)}\vol_{k\omega_{\phi}},
\end{split}\\
\begin{split}\label{di}
\left(\bold{d}\,I_{(\hxi,p)}^{k}\right)_H(\dot{H})=&\sum_{\lambda\in\Lambda_k(\hxi)}\lambda(p)\Tr(H^{-1}_\lambda\dot{H}_\lambda),
\end{split}
\end{align}
where $\phi\in\mathcal{K}^{\xi}_{\omega}$ and $H=(H_\lambda)_{\lambda\in\Lambda_k(\xi)}\in\mathcal{B}_{\hxi}(\mathcal{H}_k)$.\\
2. The second variation of $\mathbb{I}_{(\xi,p)}^{k}$ along a path $\phi_t\in \mathcal{K}^{\xi}_{\omega}$ is given by
\begin{equation}\label{d2I}
\dfrac{d^{2}}{dt^{2}}\mathbb{I}_{(\hxi,p)}^{k}(\phi_t)=2kC_k\int_X \left(\ddot{\phi}_t-|d\dot{\phi}_t|^{2}_{g_{\phi_t}}\right)f^{-p+1}_{(\hxi,\phi_t)}\vol_{k\omega_{\phi_t}}.
\end{equation}
3. For $\phi\in\mathcal{K}^{\xi}_{\omega}$ and $k\gg 1$, the Aubin functional $\mathbb{I}_{(\hxi,p)}^{k}$ is concave along the path $(\phi_k(t))_{t\in[0,1]}$ of $\mathcal{K}^{\xi}_{\omega}$ given by: 
\begin{equation}\label{pot}
\phi_k(t):=\phi+\frac{t}{2k}\log{\left(f^{p-1}_{(\hxi,\phi)}\rho_{(\hxi,p)}(k\phi)\right)}.
\end{equation} 
4. The variation of $(\hxi,p)$-Hilbert map $\Hilb_{(\hxi,p)}^k$ is given by,
\begin{align}
\begin{split}\label{dHilb}
&\left(\bold{d}\,\Hilb_{(\hxi,p)}^{k}\right)_\phi(\dot{\phi})(s,s^{\prime})=\\
&-\sum_{\lambda\in\Lambda_k(\hxi)}\frac{1}{\lambda(p)}\int_X(s(\lambda),s^{\prime}(\lambda))_{k\phi}[2k\dot{\phi}-(d\log f^{-p+1}_{(\hxi,\phi)},d\dot{\phi})_{g_\phi}+\Delta_{\phi}\dot{\phi}]f^{-p+1}_{(\hxi,\phi)}\vol_{k\omega_\phi},
\end{split}
\end{align}
where $\phi\in\mathcal{K}^{\xi}_\omega$ and $s,s^{\prime}\in\mathcal{H}_k$.
\end{lem}

We now consider the functionals $\mathcal{L}^{k}_{(\hxi,p)}:\mathcal{K}^{\xi}_\omega\rightarrow \mathbb{R}$ and $Z^{k}_{(\hxi,p)}:\mathcal{B}_{\hxi}(\mathcal{H}_k)\rightarrow \mathbb{R}$ defined by
\begin{align}\label{L-Z}
\begin{split}
\mathcal{L}^{k}_{(\hxi,p)}&=I_{(\hxi,p)}^{k}\circ \Hilb^{k}_{(\hxi,p)}+\mathbb{I}^{k}_{(\hxi,p)},\\
Z^{k}_{(\hxi,p)}&=\mathbb{I}_{(\hxi,p)}^{k}\circ \FS^{k}_{(\hxi,p)}+I^{k}_{(\hxi,p)},
\end{split}
\end{align}
where $I_{(\hxi,p)}^{k}$ is given by \eqref{i} and $\mathbb{I}_{(\hxi,p)}^{k}$ is given by \eqref{I}.
In what follows we will relate these functionals to the $(\hxi,p)$-balanced metrics, similarly to \cite{donaldson3, li, ST}, and we will show that they quantize the $(\hxi,p)$-Mabuchi energy.

\begin{prop}\label{L-Mab}
The $(\hxi,p)$-balanced metrics of order $k$ are critical points of the functional $\mathcal{L}^{k}_{(\hxi,p)}$. Furthermore, there exist real constants $b_k$ such that, 
\begin{equation*}
\underset{k\rightarrow\infty}{\lim}\left[\frac{2}{k^{m}}\mathcal{L}^{k}_{(\hxi,p)}+b_k\right]= \mathcal{M}_{(\hxi,p)},
\end{equation*} 
where the convergence holds in the $C^{\infty}$-norm.
\end{prop}
\begin{proof}
For a $\Hilb_{(\hxi,p)}^k$-orthonormal adapted basis $\{s_{i,\lambda}\mid \lambda\in\Lambda_k(\hxi),\,\,i=1,\cdots, n_k(\lambda)\}$ of $\mathcal{H}_k$, by \eqref{di} and \eqref{dHilb} we have,
\begin{align*}
&\bold{d}\left(I_{(\hxi,p)}^{k}\circ\Hilb_{(\hxi,p)}^{k}\right)_\phi(\dot{\phi})\\=& -\sum_{i,\lambda}\int_X |s_{i,\lambda}|_{k\phi}^{2}[2k\dot{\phi}-(d\log f^{-p+1}_{(\hxi,\phi)},d\dot{\phi})_{g_\phi}+\Delta_{\phi}\dot{\phi}]f^{-p+1}_{(\hxi,\phi)}\vol_{k\omega_\phi}\\
&=\int_X \rho_{(\hxi,p)}(k\phi)[2k\dot{\phi}-(d\log f^{-p+1}_{(\hxi,\phi)},d\dot{\phi})_{g_\phi}+\Delta_{\omega_{\phi}}\dot{\phi}]\vol_{k\omega_\phi}\\
&=-2k\int_X \dot{\phi}\left(1+\frac{\Delta_{\phi}}{2k}\right)\rho_{(\hxi,p)}(k\phi)\vol_{k\phi}+\int_X \rho_{(\hxi,p)}(k\phi)(d\log f^{-p+1}_{(\hxi,\phi)},d\dot{\phi})_{g_\phi}\vol_{k\omega_\phi}.
\end{align*}
By \eqref{dI} we get
\begin{align*}
\left(\bold{d}\mathcal{L}_{(\hxi,p)}^{k}\right)_\phi(\dot{\phi})=&-2k\int_X \dot{\phi}\left(1+\frac{\Delta_{\phi}}{2k}\right)\left[\rho_{(\hxi,p)}(k\phi)-C_k f^{-p+1}_{(\hxi,\phi)}\right]\vol_{k\omega_\phi}\\&+\int_X \left[\rho_{(\hxi,p)}(k\phi)-C_k f^{-p+1}_{(\hxi,\phi)}\right](d\log f^{-p+1}_{(\hxi,\phi)},d\dot{\phi})_{g_\phi}\vol_{k\omega_\phi}.
\end{align*}
From the above expression it is clear that a $(\hxi,p)$-balanced metric of order $k$ is critical point of $\mathcal{L}_{(\hxi,p)}^{k}$. By the asymptotic expansion in \Cref{rho-exp} we get
\begin{equation*}
\int_X \left[\rho_{(\hxi,p)}(k\phi)-C_k f^{-p+1}_{(\hxi,\phi)}\right](d\log f^{-p+1}_{(\hxi,\phi)},d\dot{\phi})_{g_\phi}\vol_{k\omega_\phi}=\mathcal{O}(k^{m-1}),
\end{equation*}
and
\begin{align*}
&2k\int_X\dot{\phi}\left(1+\frac{\Delta_{\phi}}{2k}\right)\left[\rho_{(\hxi,p)}(k\phi)-C_k f^{-p+1}_{(\hxi,\phi)}\right]\vol_{k\omega_\phi}\\=&2k^{m}\int_X\accentset{\circ}{{\rm Scal}}_{(\hxi,p)}(\phi)\dot{\phi}f^{-p+1}_{(\hxi,\phi)}\vol_{\omega_\phi}+\mathcal{O}(k^{m-1})\\
=&2k^{m}\left(\bold{d}\mathcal{M}_{(\hxi,p)}\right)_\phi(\dot{\phi})+\mathcal{O}(k^{m-1}).
\end{align*}
The proof is complete. 
\end{proof}

\begin{rem}
There is a natural extension of the momentum map interpretation of balanced Fubini-Study metrics given by S.~K.~Donaldson in \cite{donaldson0} to $(\hxi,p)$-balanced metrics. Indeed, let us identify $\mathcal{B}_{\hxi}(\mathcal{H}_k)$ with the space of bases of $\mathcal{H}_k$ compatible with the splitting \eqref{decomposition}, and denote by ${\rm Aut}^{\hxi}(X,L)$ the Lie group of automorphisms of the pair $(X,L)$ that commutes with the flow of $\hat{\xi}$. Let $\theta_{(\hxi,k)}$ denote the group representation of  ${\rm Aut}^{\hxi}(X,L)$ in ${\rm GL}(\mathcal{H}_k)$, given by 
\begin{equation*}
\theta_{(\hxi,k)}(\gamma)s:=\gamma\circ s\circ p(\gamma)^{-1},
\end{equation*}
where $p:{\rm Aut}(X,L)\rightarrow {\rm Aut}_{\rm red}(X)$ is the natural projection.
For each $k$ we have the following group actions on $\mathcal{B}_{\hxi}(\mathcal{H}_k)$: 
\begin{itemize}
\item[$\bullet$] $\mathbb{C}^{*}$ by scalar multiplications;
\item[$\bullet$]$\mathcal{A}_{(\hxi,k)}:=\theta_{(\hxi,k)}\left({\rm Aut}^{\hxi}(X,L)\right)$;
\item[$\bullet$]$\mathcal{G}_{(\hxi,k)}:=\left\{H\in\prod_{\lambda\in\Lambda_k(\hxi)} \mathbb{U}(\mathcal{H}_k(\lambda))\,\mid\, \prod_{\lambda}\det(g_\lambda)^{\lambda(p)}  =1\right\}$,
\end{itemize}
We consider the quotient space,
\begin{equation*}
\mathcal{Z}_{\hxi}(\mathcal{H}_k)=\bigslant{\mathcal{B}_{\hxi}(\mathcal{H}_k)}{\left(\mathbb{C}^{*}\times\mathcal{A}_{(\hxi,k)}\right)},
\end{equation*}
on which we have a natural action of $\mathcal{G}_{(\hxi,k)}$. The quotient $\mathcal{Z}_{\hxi}(\mathcal{H}_k)$ carries a natural K\"ahler structure, defined as follows:
\begin{itemize}
\item[$\bullet$]The multiplication by $\sqrt{-1}$ defines an integrable complex structure on $\mathcal{B}_{\hxi}(\mathcal{H}_k)$ invariant under the action of $\mathbb{C}^{*}\times \mathcal{A}_{(\hxi,k)}$, so it descends to a complex structure $J^{(k)}_\mathcal{Z}$ on the quotient $\mathcal{Z}_{\hxi}(\mathcal{H}_k)$.
\item[$\bullet$] There is a natural K\"ahler form on $\mathcal{B}_{\hxi}(\mathcal{H}_k)$ given by 
\begin{equation*}
\varpi^{(k)}_\mathcal{B}:=dd^{c}Z^{k}_{(\hxi,p)},
\end{equation*}
where $d^{c}:=J^{(k)}_\mathcal{B}d$. The form $\varpi^{(k)}_\mathcal{B}$ is invariant under the group actions of $\mathbb{C}^{*}\times \mathcal{A}_{(\hxi,k)}$ and $\mathcal{G}_{(\hxi,k)}$, so it defines a $\mathcal{G}_{(\hxi,k)}$-invariant K\"ahler form on $\mathcal{Z}_{\hxi}(\mathcal{H}_k)$.
\end{itemize}
We endow ${\rm Lie}(\mathcal{G}_{(\hxi,k)})$ with the pairing
\begin{equation*}
\langle a,b\rangle_{(\hat{\xi},p,k)}=\sum_{\lambda\in\Lambda(\hxi)}\lambda(p)\cdot\Tr\left(a_\lambda b^{\star}_\lambda\right),
\end{equation*}
and identify ${\rm Lie}(\mathcal{G}_{(\hxi,k)})$ with the dual vector space by using $\langle \cdot,\cdot\rangle_{(\hat{\xi},p,k)}$. The action of $\mathcal{G}_{(\hxi,k)}$ on $\mathcal{Z}_{\hxi}(\mathcal{H}_k)$ is Hamiltonian with $\varpi^{(k)}_\mathcal{Z}$-moment map $\underline{\mathcal{M}}_{(\hxi,p)}^{(k)}:\mathcal{Z}_{\hxi}(\mathcal{H}_k)\rightarrow {\rm Lie }(\mathcal{G}_{(\hxi,k)})$ given by
\begin{equation*}
\underline{\mathcal{M}}_{(\hxi,p)}^{(k)}(\textbf{s}):=\sqrt{-1}\left(\bigoplus_{\lambda\in\Lambda(\hxi)}\left(\Hilb_{(\hxi,p)}^k\left(\FS^{k}_{(\hxi,p)}(\textbf{s})\right)(s_{i,\lambda},s_{j,\lambda})\right)_{i,j=1,n(\lambda)}\right)_0
\end{equation*}
where for any $\textbf{s}\in\mathcal{B}_{\hxi}(\mathcal{H}_k)$ we identify $\textbf{s}$ with the unique positive definite Hermitian form so that $\textbf{s}$ is orthonormal, and for any $a\in {\rm Lie}(\mathcal{G}_{(\hxi,k)})$, 
\begin{equation*}
(a)_0=a-\frac{\langle a,\Id\rangle_{(\hat{\xi},p,k)}}{\langle \Id,\Id\rangle_{(\hat{\xi},p,k)}}\Id.
\end{equation*}
Thus the zeroes of the moment map $\underline{\mathcal{M}}_{(\hxi,p)}^{(k)}$ are the $(\hxi,p)$-balanced elements of $\mathcal{Z}_{\hxi}(\mathcal{H}_k)$.  
\end{rem}

\begin{lem}\label{Z-convex}
For all $\phi\in\mathcal{K}^{\xi}_\omega$ we have (in the $C^{\infty}$ sense),
\begin{equation}\label{Z-L=0}
\underset{k\rightarrow\infty}{\lim}k^{-m}\left[\mathcal{L}_{(\hxi,p)}^{k}(\phi)-Z_{(\hxi,p)}^{k}\circ \Hilb_{(\hxi,p)}^k(\phi)\right]=0.
\end{equation}
The functional $Z^{k}_{(\hxi,p)}$ is convex along the geodesics of $\mathcal{B}_{\hxi}(\mathcal{H}_k)$.
\end{lem}
\begin{proof}
The proof of the above Lemma is identical to the arguments of \cite[Propostion 3.2.3]{ST}, and \cite[Propostion 1]{donaldson3}.
\end{proof}

\begin{cor}\label{Z-bounded}
$(\hxi,p)$-balanced metrics of order $k$ minimizes the functional $Z^{k}_{(\hxi,p)}$ on $\mathcal{B}_{\hxi}(\mathcal{H}_k)$.
\end{cor}
\begin{proof}
Let $H(t),t\in\mathbb{R}$ be a geodesic in $\mathcal{B}_{\hxi}(\mathcal{H}_k)$ with $H(0)=H$ and such that $h=\FS_{(\hxi,p)}^{k}(H)$ is $(\hxi,p)$-balanced. For a choice of an adapted $H$-orthonormal basis of $\mathcal{H}_k$ denoted by $\{s_{i,\lambda}\mid \lambda\in\Lambda_k(\hxi),\,\,i=1,\cdots, n(\lambda)\}$ we have the following expression for $H(t)$
\begin{equation*}
H(t)={\rm diag}\left(e^{tA_\lambda}\right)_{\lambda\in\Lambda_k(\hxi)},
\end{equation*} 
with $A_\lambda={\rm diag}(a_i(\lambda))_{i=1,n(\lambda)}$, $a_i(\lambda)\in\mathbb{R}$ and $\Tr(A_\lambda)=0$. We consider the family of K{\"a}hler potentials given by $\phi(t):=\FS_{(\hxi,p)}^{k}(H(t))$. The collection $\{e^{\frac{-ta_i(\lambda)}{2}}s_{i,\lambda}\mid \lambda\in\Lambda_k(\hxi),\,\,i=1,\cdots, n(\lambda)\}$ is an $H(t)$-orthonormal adapted base of $\mathcal{H}_k$, so  we have
\begin{equation*}
Z_{(\hxi,p)}^{k}(t):=Z_{(\hxi,p)}^{k}(H(t))=\mathbb{I}_{(\hxi,p)}^{k}(\phi(t)).
\end{equation*}
Using the fact that $\sum_{\lambda\in\Lambda_k(\hxi)}\sum_{i=0}^{n(\lambda)}|s_{i,\lambda}|_{h^{k}}^{2}=1$ (because $h=\FS_{(\hxi,p)}^{k}(H)$) we get,
\begin{equation*}
\dot{\phi}=-\frac{1}{2k}\sum_{\lambda\in\Lambda_k(\hxi)}\sum_{i=0}^{n(\lambda)}a_i(\lambda)|s_{i,\lambda}|_{h^{k}}^{2}.
\end{equation*}
If the Hermitian metric $h^{k}$ on $L^{k}$ corresponds to a $(\hxi,p)$-balanced metric $\phi\in\mathcal{K}_{\omega}^{\xi}$ of order $k$, we have,
\begin{eqnarray*}
\dfrac{d Z_{(\hxi,p)}^{k}}{dt}(0)&=&-C_k\int_X\sum_{i,\lambda}a_i(\lambda)|s_{i,\lambda}|_{h^{k}}^{2}f^{-p+1}_{(\hxi,\omega)}\vol_{k\omega}\\
&=&-C_k\sum_{i,\lambda}a_i(\lambda)\lambda(p)\parallel s_{i,\lambda}\parallel_{\Hilb_{(\hxi,p)}^{k}(\FS_{(\hxi,p)}^{k}(H))}^{2}\\
&=&-C_k\sum_{\lambda\in\Lambda_k(\hxi)}\lambda(p)\Tr(A_\lambda)=0.
\end{eqnarray*}
Thus, $H$ is a critical point of $Z_{(\hxi,p)}^{k}$ and by the convexity of $Z_{(\hxi,p)}^{k}$, $H$ is a minimum. 
\end{proof}

Now suppose that $\mathcal{K}^{\xi}_\omega$ contains a metric $\phi^{\star}$ with $(\hxi,p)$-scalar curvature. We will show in the following proposition that the metrics $\Hilb_{(\hxi,p)}^{k}(\phi^{\star})$ are almost balanced in the sense that they minimizes $Z_{(\hxi,p)}^{k}$, up to an error that goes to zero.
\begin{prop}\label{ZH}
For all $\phi\in\mathcal{K}^{\xi}_\omega$ there exists a smooth function $\varepsilon_\phi(k)$, such that $\underset{k\rightarrow\infty}{\lim}\varepsilon_\phi(k)=0$ in $C^{\ell}(X,\mathbb{R})$ and, 
\begin{equation*}
k^{-m}Z_{(\hxi,p)}^{k}\circ \Hilb_{(\hxi,p)}^{k}(\phi)\geq k^{-m}Z_{(\hxi,p)}^{k}\circ \Hilb_{(\hxi,p)}^{k}(\phi^{\star})+\varepsilon_\phi(k).
\end{equation*}
\end{prop}
\begin{proof}
We denote $H_k=\Hilb_{(\hxi,p)}^k(\phi)$ and $H^{\star}_k=\Hilb_{(\hxi,p)}^k(\phi^{\star})$. For a choice of an adapted $H^{\star}_k$-orthonormal basis $\{s_{i,\lambda}\}$ of $\mathcal{H}_k$ we can write $H_k={\rm diag}(e^{A_\lambda})_{\lambda\in\Lambda_k(\hxi)}$ with $A_\lambda={\rm diag}(a_i(\lambda))_{i=1,n(\lambda)}$, $\Tr(A_\lambda)=0$, and consider the geodesic that joins $H^{\star}_k$ to $H_k$,
\begin{equation*}
H_k(t)={\rm diag}\left(e^{tA_\lambda}\right)_{\lambda\in\Lambda_k(\hxi)}.
\end{equation*} 
Let $P_k(t):=Z_{(\hxi,p)}^{k}(H_k(t))$. $P_k(t)$ is a convex function by \Cref{Z-convex}. It follows that,
\begin{equation*}
k^{-m}\left(Z_{(\hxi,p)}^{k}(H_k)-Z_{(\hxi,p)}^{k}(H^{\star}_k)\right)\geq k^{-m}P^{\prime}_k(0).
\end{equation*}
Letting $\varepsilon_\phi(k):=k^{-m}P^{\prime}_k(0)$, we have 
\begin{equation*}
P^{\prime}_k(0)=2kC_k\int_X\dot{\phi}f_{(\hxi,\phi^\star)}^{-p+1}\vol_{k\omega_{\phi^\star}}=-C_k\int_X\frac{\rho_A(k\phi^\star)}{\rho_{(\hxi,p)}(k\phi^\star)}f_{(\hxi,\phi^\star)}^{-p+1}\vol_{k\omega_{\phi^\star}},
\end{equation*}
where
\begin{equation}\label{s}
\rho_A(k\phi^\star)=\sum_{i,\lambda}a_i(\lambda)f_{(\hxi,\phi^\star)}^{-p+1}|s_{i,\lambda}|_{k\phi^\star}^{2}.
\end{equation}
By \Cref{rho-exp}, since $\phi^{\star}$ has constant $(\hxi,p)$-scalar curvature we get
\begin{equation}\label{ss}
\rho_{(\hxi,p)}(k\phi^\star)=f_{(\hxi,\phi^\star)}^{-p+1}+\mathcal{O}(k^{-2}),
\end{equation}
and therefore we obtain
\begin{equation*}
P^{\prime}_k(0)=-C_k\int_X\rho_A(k\phi^\star)\mathcal{O}(k^{-2})\vol_{k\omega_{\phi^\star}}.
\end{equation*}
We have
\begin{equation}\label{eq-T}
e^{a_i(\lambda)}=\parallel s_{i,\lambda}\parallel_{H_k}^{2}=\sum_{\lambda\in\Lambda_k(\hxi)}\lambda(p)^{-1}\int_X|s_{i,\lambda}|_{k\phi}^{2}f_{(\hxi,\phi)}^{-p+1}\vol_{k\omega_\phi}.
\end{equation}
As $h_{\phi}^k=e^{-2k(\phi^{\star}-\phi)}h_{\phi^{\star}}^k$, there exists a constant $C_\phi>0$ such that 
\begin{equation}\label{est}
e^{-2kC_\phi}h_{\phi^{\star}}^k\leq h_{\phi}^k\leq e^{2kC_\phi}h_{\phi^{\star}}^k.
\end{equation}
By the fact that $f_{(\hxi,\phi)}^{-p+1}/f_{(\hxi,\phi^{\star})}^{-p+1}$ is bounded by positive constants (independent from $\phi$), and $\vol_{k\omega_\phi}/\vol_{k\omega_{\phi^\star}}$ is bounded by positive constants depending only on $\phi$, using \eqref{est} we obtain from \eqref{eq-T} the following estimate
\begin{equation}\label{sss}
-2C_\phi k+ B^{\prime}_\phi\leq a_i(\lambda)\leq 2C_\phi k+ B_\phi,
\end{equation}
where $B_\phi, B^{\prime}_\phi$ are real constants depending only on $\phi,\phi^{\star}$. We derive from \eqref{s} and \eqref{sss} that, 
\begin{equation*}
(-2C_\phi k+B^{\prime}_\phi)\rho_{(\hxi,p)}(k\phi^\star) \leq \rho_A(k\phi^\star)\leq (2C_\phi k+ B_\phi)\rho_{(\hxi,p)}(k\phi^\star).
\end{equation*}
Using \eqref{ss} we infer
\begin{equation*}
(-2C_\phi k+B^{\prime}_\phi)f_{(\hxi,\phi^\star)}^{-p+1}+\mathcal{O}(k^{-1}) \leq \rho_A(k\phi^\star)\leq (2C_\phi k+ B_\phi)f_{(\hxi,\phi^\star)}^{-p+1}+\mathcal{O}(k^{-1}),
\end{equation*}
which shows that $\underset{k\rightarrow\infty}{\lim}\varepsilon_\phi(k)=0$.
\end{proof}

\subsection{Proof of \Cref{Mabuchi-bounded}}
Now we are in position to give the proof of \Cref{Mabuchi-bounded} which is very similar to \cite[Theorem 3.4.1]{ST}.
\begin{proof}
Let $\phi^{\star}\in\mathcal{K}^{\xi}_{\omega}$ the K\"ahler potential of a metric with constant $(\hxi,p)$-scalar curvature. For any $\phi\in\mathcal{K}^{\xi}_{\omega}$, by \Cref{Z-bounded} we have
\begin{eqnarray*}
\mathcal{L}^{k}_{(\hxi,p)}(\phi)&=&Z^{k}_{(\hxi,p)}(\Hilb_{(\hxi,p)}^k(\phi))+\left[\mathcal{L}^{k}_{(\hxi,p)}(\phi)-Z^{k}_{(\hxi,p)}(\Hilb_{(\hxi,p)}^k(\phi))\right]\\
&\geq& Z_{(\hxi,p)}^{k}( \Hilb_{(\hxi,p)}^k(\phi^{\star}))+k^{m}\varepsilon_\phi(k)+\left[\mathcal{L}^{k}_{(\hxi,p)}(\phi)-Z^{k}_{(\hxi,p)}(\Hilb_{(\hxi,p)}^k(\phi))\right].
\end{eqnarray*}
Thus,
\begin{eqnarray*}
\frac{2}{k^{m}}\mathcal{L}^{k}_{(\hxi,p)}(\phi)+b_k&\geq& \frac{2}{k^{m}}\mathcal{L}^{k}_{(\hxi,p)}(\phi^{\star})+b_k+\frac{2}{k^{m}}\left[Z_{(\hxi,p)}^{k}( \Hilb_{(\hxi,p)}^k(\phi^{\star}))- \mathcal{L}^{k}_{(\hxi,p)}(\phi^{\star})\right]\\&&+\varepsilon_\phi(k)+\frac{2}{k^{m}}\left[\mathcal{L}^{k}_{(\hxi,p)}(\phi)-Z^{k}_{(\hxi,p)}(\Hilb_{(\hxi,p)}^k(\phi))\right].
\end{eqnarray*}
Using \Cref{L-Mab} and \ref{ZH} together with \eqref{Z-L=0}, by letting $k$ go to infinity we get,
\begin{equation*}
\mathcal{M}_{(\hxi,p)}(\phi)\geq\mathcal{M}_{(\hxi,p)}(\phi^{\star}).
\end{equation*}
\end{proof}

\section{The conformally K\"ahler, Einstein--Maxwell metrics on ruled surfaces}
In this section, we give the proof of the \Cref{c:main} from the Introduction. 
\subsection{The Calabi construction of cKEM metrics on ruled surfaces}  Let $X= \mathbb{P}(\cO \oplus \cL) \to C$ be a geometrically ruled complex surface over a compact complex curve $C$ of genus ${\bf g}\geq 2$. Following \cite{KoTo16}, cKEM metrics can be constructed by using the Calabi ansatz~\cite{acgt, calabi, hwang-singer}:  Let $(g_C, \omega_C)$ be a  K\"ahler metric on $C$ with constant scalar curvature $4(1-{\bf g})$,  where $\ell = {\rm deg}(\cL) >0$ is the degree of $\cL$. We denote by $\theta$ the connection $1$-form on the principal ${\mathbb S}^1$-bundle $P$ over $C$,  with curvature $d\theta =  \ell \omega_C$. Notice that $P$ can be identified with the unitary bundle of $(\cL^*, h^*)$ over $C$, where $h^*$ is the Hermitian metric with Chern curvature $-\ell \omega_C$; viewing equivalently  $X$ as a compactification at infinity of $\cL^* \to C$  (i.e. $X=\mathbb{P}(\cL^* \oplus \cO)$), we can introduce a class of K\"ahler metrics on $X$ by
\begin{equation}\label{calabi-ansatz}
g = \ell (z + \kappa) g_C + \frac{dz^2}{\Theta(z)} + \Theta(z) \theta^2, \, \omega= \ell(z+\kappa)\omega_C + dz\wedge \theta,
\end{equation}
where:  $z \in [-1,1]$ is a momentum variable for the ${\mathbb S}^1$-action on $\cL^*$, $\Theta(z)$ is a smooth function on $[-1,1]$, called  a {\it profile function} \cite{hwang-singer},  which satisfies the first order boundary conditions 
\begin{equation}\label{first-order}
\Theta(\pm 1)=0, \ \Theta'(\pm 1)= \mp 2, 
\end{equation} 
and the positivity condition 
\begin{equation}\label{positivity}
\Theta(z)>0 \ {\rm on} \ (-1,1).
\end{equation}
Here $\kappa>1$ is a real constant which parametrizes the  K\"ahler class 
$$\Omega_\kappa = [\omega] = 2\pi \Big(c_1(\cO(2)_{\mathbb{P}(\cO \oplus \cL)})  +(1+ \kappa)\ell [\omega_C]\Big).$$
Notice that for the ruled surfaces we consider $H^2(X, \R) \cong   \R^2$, so that any K\"ahler class on $X$ can be written as $\lambda \Omega_\kappa$ for some $\lambda>0$ and $\kappa>1$, see \cite{fujiki}. Furthermore, $\Omega_\kappa$ is homothetic to a Hodge class if and only if $\kappa\in(1,+\infty)\cap\mathbb{Q}$.\\

\smallskip

For any $|b|>1$, $f=|z+b|$ is a positive Killing potential with respect to \eqref{calabi-ansatz}  which  corresponds up to sign to the  Killing vector field $\xi$ generating the $\mathbb{S}^1$-action on $X= \mathbb{P}(\cO \oplus \cL)$ by multiplications of the  first factor $\cO$; the additive constant $b$ is simply an affine-linear modification of the normalizing constant $a$ for the Killing potentials of $\xi$. The main results  of  \cite{KoTo16}  can be summarized as follows

\begin{prop}\cite{KoTo16}\label{Koca-Tonnesen-Friedman} Let $X=\mathbb{P}(\cO \oplus \cL)\to C$ be a ruled complex surface as above. 
\begin{enumerate}
\item[$\bullet$] For any $\kappa>1$, the Futaki invariant $\mathfrak{F}_{(\Omega_\kappa, \xi, b)}$ of \cite{AM} vanishes if and only if $b$ satisfies
\begin{equation}\label{futaki=0}
\kappa= \frac{1+b^2}{2b}.
\end{equation}
We denote by $b_\kappa>1$ the unique solution of \eqref{futaki=0} satisfying $|b|>1$.
\item[$\bullet$] There exits a polynomial $P_{\kappa}(z)$ of degree $\leq 4$ such that $\Theta(z) =P_{\kappa}(z)/(z+\kappa)$ satisfies the first order boundary conditions \eqref{first-order} and, on any open subset when $\Theta(z)>0$, the metric \eqref{calabi-ansatz} is conformal to a cKEM metric with conformal factor $(z+b_\kappa)^{-2}$.
\item[$\bullet$] There exists  $\kappa_0(X)  \in (1,+\infty)$  such that 
\begin{enumerate}
\item[\rm (a)] for each $\kappa \in (\kappa_0(X), +\infty)$ the corresponding polynomial $P_{\kappa}(z)>0$ on $(-1,1)$, i.e. $\Omega_\kappa$ admits a K\"ahler metric of the form \eqref{calabi-ansatz} with $\Theta(z)=P_\kappa(z)/(z+\kappa)$, such that $(z+ b_\kappa)^{-2}g$ is cKEM;
\item[\rm (b)] for each $\kappa \in (1,\kappa_0(X))$ the corresponding polynomial $P_{\kappa}(z)$ is negative somewhere on $(-1,1)$;
\item[\rm (c)] for $\kappa=\kappa_0(X)$ the corresponding polynomial $P_{\kappa}(z)$ is non-negative and has a zero with multiplicity $2$ on $(-1,1)$.
\end{enumerate}
\end{enumerate}
\end{prop}

\subsection{The $(\xi, b_\kappa,4)$-Mabuchi energy} For a given $\kappa>0$, the K\"ahler metrics \eqref{calabi-ansatz} define a Frech\'et subspace ${\mathcal Cal}_{\omega_\kappa}$, modelled on smooth profile functions $\Theta(z)$ satisfying \eqref{first-order} and \eqref{positivity}, of $\omega_\kappa$-compatible $\xi$-invariant K\"ahler metrics on $X$. We can choose as a reference metric $(g_\kappa, \omega_\kappa)$ of the form \eqref{calabi-ansatz} with profile function $\Theta_0:=(1-z^2)$; the corresponding complex structure $J$ can be naturally identified with the canonical complex structure on $X$, see e.g. \cite{acgt}. Letting $u''(z)= 1/\Theta(z)$ be {\it the fibre-wise symplectic potential} of a metric in ${\mathcal Cal}_{\omega_\kappa}$, the fibre-wise Legendre transform of $u$ defines a differentiable map $\mathcal T : {\mathcal Cal}_{\omega_\kappa} \to {\mathcal K}_{\Omega_\kappa}^{{\mathbb S}^1}$ with differential ${\boldsymbol d} {\mathcal T}_g(\dot u) = - \dot \phi$ see \cite{acgt}. Thus, a positive multiple of the pull-back of ${\mathcal M}_{(\xi, b, 4)}$ to ${\mathcal Cal}_{\omega_\kappa}$ is defined by
\begin{equation*}
\begin{split}
({\boldsymbol d} {\mathcal M}_{(\xi, b, 4)})_g(\dot u)&= \int_{-1}^{1} \dot u(z) {\mathring{{\rm Scal}}}_{(\xi, b, 4)} |z+ b|^{-5} dz \\
{\mathcal M}_{(\xi, b, 4)}({g_\kappa}) &=0.
\end{split}
\end{equation*}
A computation similar to \cite[Prop.~7]{acgt} reveals that for $b=b_\kappa$, the  solution is given by the formula
\begin{prop}\cite{amt}\label{amt} The Mabuchi energy ${\mathcal M}_{(\xi, b_\kappa, 4)}$ restricted to the space ${\mathcal Cal}_{\omega_\kappa}$ of K\"ahler metrics given by the Calabi ansatz \eqref{calabi-ansatz} is up to a positive constant 
\begin{equation*}
\begin{split}
\mathcal{M}_{(\xi, b_\kappa, 4)} : u(z) \longmapsto & \int_{-1}^1 \frac{P_{\kappa}(z)}{(z+b_\kappa)^{3}}\big(u''(z)-u''_\kappa(z)\big)dz \\
                                                           &-\int_{-1}^{1}\frac{(z+\kappa)}{(z+b_\kappa)^{3}}\log\Big(\frac{u''(z)}{u''_\kappa(z)}\Big)dz,
                                                      \end{split}
                                                      \end{equation*}
                                                      where $u_\kappa(z)= \frac{1}{2}\Big((1-z) \log (1-z) + (1+z) \log (1+z)\Big)$ is the fibre-wise symplectic potential of the canonical metric $g_\kappa\in {\mathcal Cal}_{\omega_\kappa}$ and $b_\kappa$ and $P_\kappa(z)$ are the real number and polynomial of Proposition~\ref{Koca-Tonnesen-Friedman}.
\end{prop}
As noticed in the proof of \cite[Cor.~3]{acgt}, if $P_\kappa(z)$ is negative on an interval $I \subset (-1,1)$, taking $f(z)$ to be a bump function with support in $I$    and $u_k(z)$ a sequence of symplectic potentials defined by $u_k''(z)= u_\kappa ''(z) + kf(z), k >0$, one has $\lim_{k\to \infty} \mathcal{M}_{(\xi, b_\kappa, 4)} (u_k)= -\infty$. We thus get
\begin{cor}\label{Lahdili} If the polynomial $P_\kappa(z)$ given in Proposition~\ref{Koca-Tonnesen-Friedman} is negative somewhere on $(-1,1)$, then the Mabuchi functional $\mathcal{M}_{(\xi, b_\kappa, 4)} $ is not bounded from below.
\end{cor}                          

\subsection{Proof of Corollary~\ref{c:main}}  There are no cscK metrics on $X$ (see e.g. \cite{at}), so that we are looking for strictly conformally K\"ahler, Einstein--Maxwell metrics. As in our case ${\rm Aut}_{\rm red}(X, J)=\C^*$ (see e.g. \cite{at}), the Killing vector field $\xi$ must be a multiple of the vector field  generating rotations on the factor $\cO$. As the theory is invariant under homothety of the Killing potential, without loss we assume that this multiple is $\pm 1$.  Finally, as $H^2(X, \R)= \R^2$, by rescaling the K\"ahler class we can also assume $\Omega= \Omega_\kappa, \kappa>1$. For a K\"ahler metric $g \in \Omega_\kappa$ of the form \eqref{calabi-ansatz}, the Killing potential of $\xi$ is $|z+b|$ with $|b|>1$. The necessary condition $\mathfrak{F}_{\Omega_\kappa, \xi, b}=0$ then forces us to consider $b=b_\kappa$, see Proposition~\ref{Koca-Tonnesen-Friedman}. The existence of conformally K\"ahler, Einstein--Maxwell metrics for $\kappa\in (\kappa_0(X), \infty)$ and conformal factor $(z+b_\kappa)^{-2}$  follows from the statement in (a) of Proposition~\ref{Koca-Tonnesen-Friedman}. 

We are left to show non-existence for $\kappa\in (1,\kappa_0(X)]$. Again, by Proposition~\ref{Koca-Tonnesen-Friedman}, we have to take $b=b_\kappa>1$. 

Consider first the case $\kappa\in (1,\kappa_0(X))$. If $\kappa$ is rational, the result follows from Theorem~\ref{Mabuchi-bounded} and Corollary~\ref{Lahdili}. Otherwise, if $\kappa\in (1,\kappa_0(X)) \setminus \Q$, we suppose for contradiction that $\Omega_\kappa$ admits a K\"ahler metric of constant  $(z+b_\kappa, 4)$-scalar curvature. By \cite[Theorem 2]{lahdili}, the same will hold for all $(\kappa', b_{\kappa'})$ on the rational curve \eqref{futaki=0} which are sufficiently close to $(\kappa, b_\kappa)$, in particular  for all rational pairs $(\kappa', b_{\kappa'})$ close to $(\kappa, b_\kappa)$, a contradiction. Finally, consider $\kappa=\kappa_0(X)=\kappa_0$, $b_{\kappa_0}= b_0$. Again, suppose for contradiction that $\Omega_{\kappa_0}$ admits a metric of constant $(\xi,b_0,4)$-scalar curvature. We use again \cite[Theorem 2]{lahdili} to deduce that this holds also for all $(\kappa, b_\kappa)$ near $(\kappa_0, b_0)$ and we can find again {\it rational} valued  $(\kappa, b_\kappa)$ arbitrarily close to $(\kappa_0, b_0)$ with $\kappa<\kappa_0$, and still admitting a $(\xi,b_\kappa, 4)$-extremal K\"ahler metric, a contradiction.

\end{document}